\documentclass{article}
\usepackage{graphicx}
\usepackage{amsmath, amssymb,amsthm}
\usepackage[margin=1in]{geometry}

\numberwithin{equation}{section}
\newtheorem{definition}{Definition}[section]
\newtheorem{theorem}{Theorem}[section]
\newtheorem{proposition}{Proposition}[section]
\newtheorem{lemma}{Lemma}[section]
\newtheorem{remark}{Remark}[section]
\newtheorem{example}{Example}[section]

\begin{document}

\title{Duality and Complete Convergence for Multi-Type Additive Growth Models}

\author{Eric Foxall}

%\institute{
%              Department of Mathematics and Statistics, University of Victoria \\
%			  PO BOX 3060 STN CSC Victoria, B.C. Canada V8W 3R4 \\
 %             Tel.: (250) 208-2283\\
  %            \email{e.t.foxall@gmail.com}
%}
%\date{Received: date / Accepted: date}

\maketitle

\begin{abstract}
We consider a class of multi-type particle systems having similar structure to the contact process and show that additivity is equivalent to the existence of a dual process, extending a result of Harris.  We give two additional characterizations of these systems, in spacetime as percolation models, and biologically as population models in which the interactions are due to crowding.  We prove a necessary and sufficient condition for the model to preserve positive correlations.  We then show that complete convergence on $\mathbb{Z}^d$ holds for a large subclass of models including the two-stage contact process and a household model, and give examples.
%\keywords{Contact process \and Additive Process \and Interacting Particle Systems}
%\subclass{60J25 \and 92B99}
\end{abstract}

\section{Introduction}
The contact process \cite{ips},\cite{sis} is a model of infection spread in which each individual is in one of two states, either healthy or infectious.  The model has a number of convenient properties, namely, additivity and preservation of positive correlations \cite{gc} which in many settings allow it to be analysed more or less completely; see for example the proof, give in \cite{sis}, of complete convergence for the model on the lattices $\mathbb{Z}^d$, which characterizes not only the set of extremal invariant measures as a function of the infection parameter but also the convergence to a combination of these measures, from any initial configuration.  Since the proof does not rely on other details of the model, for example, the number of infectious types, it seems natural to suppose that the same results should hold for a much wider class of models.\\

To this end, we consider a class of particle systems that we refer to as \emph{growth models}.  The term is borrowed from \cite{growth}, however here we work in continuous rather than discrete time, and with systems having generally more than one type of active particle.  Precise definitions are given below, but we note that the class of models considered includes the two-stage contact process \cite{krone}, \cite{fox} the household model studied in \cite{konnohousehold} and a spatial analogue of any multi-type branching process (see \cite{bookmtbp} for an introduction to these processes).  Since we focus on additive systems, our class of models does not include, for example, the multitype contact process from \cite{mcp}.  In \cite{gc}, Harris showed that for contact processes that additivity is equivalent to the existence of a dual process; we generalize this result to systems with multiple active types, and give some additional descriptions of these processes, as Harris did for contact processes.  Also, we show for a nicely behaved subclass of these models that complete convergence holds in the sense proved for stochastic growth models in \cite{growth}, and that there is a well-defined critical ``surface'' in parameter space separating survival from extinction.  In general we use $\eta_t$ to denote the process under consideration.\\

We first recall the contact process.  Given a graph with sites $S$ and edges $E$, the contact process is a continuous-time stochastic process taking place on the set of configurations $\{0,1\}^S$ in which each site $x$ is either healthy ($\eta(x)=0$) or infectious ($\eta(x)=1$).  A healthy site becomes infectious at rate $\lambda$ times the number of its infectious neighbours, and an infectious site recovers at rate 1.  The graphical construction, described in \cite{sis} and used in the next section, allows us to construct the process, using Poisson point processes with the given rates, for all sets of initial configurations on the same probability space.  We can also think of the contact process as a model for population growth, substituting vacant for healthy and occupied for infectious; this interpretation is useful when we define a notion of ``primitive'' active types, which play the role of different types of organisms in a multi-type process.\\

The contact process has a certain \emph{additivity} property which simplifies its analysis.  That is, defining $\vee$ on pairs of configurations by
\begin{equation*}
(\eta\vee\eta')(x) = \max(\eta(x),\eta'(x))
\end{equation*}
and letting $\eta_t$ and $\eta_t'$ refer to the process with initial configurations $\eta_0$ and $\eta_0'$, the process with initial configuration $\eta_0\vee\eta_0'$ is given by $\eta_t \vee \eta'_t$ for each $t>0$.  In other words, by superimposing the processes obtained from two initial configurations, we recover the process run from the superimposed initial configuration.  This property is obvious from the graphical construction; see, for example, \cite{gc}.\\

Our generalization is to processes that we call \emph{multi-type particle systems}, with a finite set of types $F = \{0,1,...,N\}$, so that the state space of configurations is $F^S$.  The local dynamics is given by a family of transition rates
\begin{equation*}
\{c_T(\phi,\psi):\phi,\psi \in F^T,T\subseteq S \textrm{ finite }\}
\end{equation*}
that give the rate at which $\eta|_T$, the restriction of $\eta$ to $T$, flips to $\psi$ when its value is $\phi$.  Under certain regularity assumptions on the rates, we can construct a stochastic process with the rates given.  These types of systems are well-studied, particularly those in which $F = \{0,1\}$; the interested reader may consult \cite{ips} for sufficient conditions on the rates so that the process is uniquely defined.  In the next section we give a simple condition for this to hold, and a graphical construction that shows existence and uniqueness.\\

We say the process is \emph{additive} if $F$ is partially ordered by a \emph{join} operation, that is, a binary operation $a\vee b$ satisfying $a \vee b \geq a,b$ and $c \geq a,b \Rightarrow c \geq a\vee b$, and such that
\begin{equation*}
\eta_0'' = \eta_0 \vee \eta_0' \Rightarrow \eta_t'' = \eta_t \vee \eta_t' \textrm{ for } t>0
\end{equation*}
with respect to the graphical construction, which clearly generalizes the definition for the contact process.  Letting $\underline{0}$ denote the configuration with all $0$'s, we say the process is a \emph{growth model} if type $0\leq a$ for all $a \in F$ and if $\underline{0}$ is absorbing and is reachable from any configuration with only finitely many sites in a non-0 state.\\

Note that the definition of additivity is not completely precise; as we shall see in Section \ref{seccoup}, given a process with well-defined transition rates there may be multiple ways of construcing it graphically, some of which may satisfy the above definition while others do not, thus in Section \ref{secattr} we shall arrive at the slightly more precise Definition \ref{adddef}.  However, the present definition is enough to understand the statement of the results.\\

We now summarize the main results.  The following result was proved in \cite{gc} for $F=\{0,1\}$, using graphical methods; we prove it here for general $F$, using a simpler and more general technique that focuses on individual events.  Roughly speaking, a dual is a process going backwards in time that gives information about the forward process; a precise definition is given in Section \ref{secgrowth}.
\begin{theorem}\label{thmdual}
Every additive growth model has a dual which is an additive growth model.
\end{theorem}
The contact process turns out to be \emph{self-dual}, i.e., the dual of the contact process is itself the contact process.  This fact allows us to prove, for example, that the contact process has \emph{single-site survival} (i.e. has a positive probability of surviving for all time, started from a single infected site - note this property does not depend on the site considered, provided the graph is connected) if and only if it has a non-trivial upper invariant measure (started from all sites infected, converges to an invariant measure that concentrates on configurations with infinitely many infected sites).  In Example \ref{ex1} we show that for each $N$ there is an $N$-stage contact process which is also self-dual, and thus has also the property just stated.\\

We turn now to showing that for a growth model, the existence of a dual implies additivity.  To do this it is useful to pass to a slightly better type of growth model.  An additive growth model has a set of \emph{primitive} types defined to be those $a \in F$ such that $a \neq b \vee c$ for any $b,c \in F$ with $b\neq a$, $c\neq a$.  If we imagine primitive types $a,b$ as individual organisms, we can think of $a\vee b$ as a state in which both an $a$ and a $b$ organism live together at a single site.  We can then expand the set of types $F$ so that it distinguishes all combinations of primitive types; a growth model that satisfies this condition is called \emph{multi-colour}.  The next result states that after expanding the set of types, we can always obtain a well-defined process that projects down to the original process.
\begin{theorem}\label{thmlift}
Every additive growth model $\eta$ with types $F$ has a \emph{lift} $\xi$ with types $F_*$ which is an additive multi-colour growth model and satisfies
\begin{equation*}
\pi(\xi_0) = \eta_0 \Rightarrow \pi(\xi_t) = \eta_t \textrm{ for }t>0
\end{equation*}
where $\pi:F_*^S\rightarrow F^S$ is given by $\pi(\xi)(x) = \pi(\xi(x))$ for some $\pi:F_*\rightarrow F$ that preserves the join.
\end{theorem}
The definition of a multi-colour growth model does not require that the model be additive.  However, for this type of model we can show that having a dual and being additive are equivalent.
\begin{theorem}\label{thmadd}
If a multi-colour growth model has a dual, then the model is additive.
\end{theorem}
Because they are defined in a natural way in terms of primitive types, additive multi-colour growth models have a couple of useful interpretations, namely
\begin{itemize}
\item as percolation models in spacetime, and
\item as population models in which organisms neither directly inhibit, nor help one another, but interact only through the effects of crowding
\end{itemize}
These interpretations are described in Sections \ref{secperc} and \ref{secpop}, and some basic facts are proved regarding them.  The percolation viewpoint has been used successfully to prove results about the contact process, as in \cite{basic}, where a contour method for proving survival (and estimating the critical value) is described.  The population viewpoint allows us to construct, from any multi-type branching process, a spatial model having the same set of transitions but with the crowding effect arising from the ``at most one particle of each type per site'' rule.\\

For an additive multi-colour growth model we characterize the property PC, or positive correlations, in terms of the individual transitions, as described in Section \ref{secPC}.  PC is a helpful property for particle systems; for example, in \cite{crit} it is a crucial property in comparing the contact process to oriented percolation.  Included in this characterization are models having only \emph{productive} or \emph{destructive} transitions, that is, models in which at each transition, either particles are produced or ``upgraded'' to stronger particles, or else particles are killed or ``downgraded'' to weaker particles, but not both.  If we attach a parameter to each type of interaction in such a model, as should be clear after the next section there is a natural sense in which the model is increasing with respect to occurrence of productive interactions and is decreasing with respect to destructive interactions.  It follows that in the resulting parameter space, if we fix all but one parameter $\lambda$ then there is a critical value $\lambda_c$ (possibly equal to $0$ or $\infty$) which is decreasing with respect to the rate of (other) productive transitions and increasing with respect to the rate of (other) destructive transitions, such that single-site survival occurs when $\lambda>\lambda_c$ and does not occur when $\lambda<\lambda_c$ - this is the so-called ``critical surface'' alluded to earlier.\\

By taking the join of all configurations we see that an additive model has a largest configuration.  If with respect to the graphical construction, $\eta_0\geq \eta_0'$ implies $\eta_t \geq \eta_t'$ for $t>0$ we say the model is \emph{attractive}; it follows easily from the definition of the partial order in terms of the join that an additive model is attractive.  For an attractive model with a largest configuration, the distribution of the process $\eta_t$ started from that configuration is decreasing in time, since for $s<t$, $\eta_0 \geq \eta_{t-s}$ and evolving by $s$ forward in time, by attractiveness $\eta_s \geq \eta_t$.  By first examining cylinder sets of configurations whose probabilities decrease in time, and then using these to determine finite-dimensional distributions, we can show the distribution of the process converges to an \emph{upper invariant measure} $\nu$.  We say the model exhibits \emph{complete convergence} if for any initial configuration $\eta_0$ we have
\begin{equation*}
\eta_t \Rightarrow (1-\sigma(\eta_0))\delta_0 + \sigma(\eta_0)\nu\\
\end{equation*}
where $\delta_0$ is the measure that concentrates on the all zero configuration and $\sigma(\eta_0) = \mathbb{P}(\eta_t \neq \underline{0} \,\,\forall t>0\}$ is the probability of survival starting from $\eta_0$, and $\Rightarrow$ denotes weak convergence of measures.  The following result is proved in Section \ref{seccomp}, and uses the construction of \cite {crit}.  Translation invariant and symmetric means that transitions are invariant under translation and under reflection in each coordinate.  Irreducible means that any active type at one site can produce any active type at any other site via some sequence of transitions.  ``On $\mathbb{Z}^d$'' just means that the set of sites is the d-dimensional integer lattice $\mathbb{Z}^d$, where $d\geq 1$ is any integer.

\begin{theorem}\label{thmcc}
An additive multi-colour growth model on $\mathbb{Z}^d$ which is irreducible, translation invariant, symmetric and has only productive and destructive interactions exhibits complete convergence.  Moreover, single-site survival is equivalent to having $\nu \neq \delta_0$.  
\end{theorem}

The paper is organized as follows.  In Section \ref{seccoup} we introduce the graphical method and describe the set of couplings that it furnishes.  In Section \ref{secattr} we introduce the notions of attractiveness and additivity.  In Section \ref{secgrowth} we define growth models and prove Theorem \ref{thmdual}.  In Section \ref{secprim} we define multi-colour systems and prove Theorems \ref{thmlift} and \ref{thmadd}.  In Sections \ref{secperc}, \ref{secpop} and \ref{secPC} we discuss percolation, population, and positive correlations.  In Section \ref{seccomp} we prove complete convergence.\\

\section{Graphical method and coupling}\label{seccoup}
We begin by introducing a graphical method with the intention of describing a large but manageable class of couplings for multi-type particle systems.  Graphical methods have been extensively applied to particle systems; see \cite{gc} for one of the first papers on the subject.\\

\begin{definition}
A \emph{local mapping} is a function $e:F^T\rightarrow F^T$ for some finite $T\subseteq S$.  An \emph{event structure} is a family of distinct local mappings $\{e_i\}$ and associated (positive) rates $\{r_i\}$ with each $e_i:F^{T_i}\rightarrow F^{T_i}$ for some $T_i$.
\end{definition}

Our goal is to turn an event structure into a process $\eta_t$.  Since $\eta_t$ is constructed from a collection of spacetime events, we call this the \emph{event construction}.  First define a family of independent Poisson point processes $\{u_i\}$ such that each $u_i$ has intensity $r_i$.  Then, place the events corresponding to the processes $\{u_i\}$ on the spacetime set $\mathcal{S} = S\times\mathbb{R}^+$ to obtain a \emph{spacetime event map}.  When an event occurs that corresponds to $u_i$, we will apply the mapping $e_i$, as we now describe.\\

If $|S|<\infty$ (i.e., the set of sites is finite) and for each $T\subseteq S$ at most finitely many mappings $e_i$ send $F^T\rightarrow F^T$, then for each $t>0$, only finitely many events occur in the interval $(0,t]$.  Thus, given the initial configuration $\eta_0$ we can determine $\eta_t$ for $t>0$ by accounting for each event.  If $u_i$ has an event at time $s$, use $\eta_{s^-}$ to denote the state just prior to the event.  Letting $\phi = \eta_{s^-}|_{T}$ be the restriction of $\eta_{s^-}$ to $T$, we let
\begin{equation*}
\eta_s|_{T} = e_i(\phi)
\end{equation*}
and $\eta_s(x) = \eta_{s^-}(x)$ for $x \notin T$.\\

If $|S|=\infty$, to determine the state at each spacetime point $(x,t) \in \mathcal{S}$ it suffices that the state of only a bounded set of points in $\mathcal{S}$ can affect the state at $(x,t)$; the following condition is sufficient for this to hold.\\

\begin{lemma}\label{bp}
Given an event structure $\{e_i\}$, $\{r_i\}$, if there is an $M$ and an $L$ so that
\begin{itemize}
\item $r_i = 0$ if $|T_i|>M$ and 
\item each $r_i \leq L$
\end{itemize}
then for each spacetime point $(x,t) \in \mathcal{S}$, the state at $(x,t)$ is only affected by the state of a bounded set of points in $\mathcal{S}$.
\end{lemma}

\begin{proof}
Looking backwards in time, the number of sites $y$ at time $t-s$ whose state may affect the state at $(x,t)$ is bounded by a branching process in which each individual gives birth to $M$ individuals at rate $L$, which is known to be finite for finite times.\\
\end{proof}

\begin{remark}
These are fairly mild conditions, since we assume only that each event involves (and depends on) at most $M$ sites, and that each type of event occurs at rate at most $L$.  Note, however, that there are well-behaved systems not satisfying the first condition; see for example the nearest particle systems of Chapter 7 in \cite{ips}.
\end{remark}

To construct a multi-type particle system from its transition rates, one way to do it is to associate to each transition $c_T(\phi,\psi)$ the mapping $e_i$ that sends $\phi$ to $\psi$ while leaving other local configurations unchanged, and the rate $r_i = c_T(\phi,\psi)$.  Letting $c_T = \sum_{\phi,\psi}c_T(\phi,\psi)$, to satisfy the conditions of Lemma \ref{bp} it suffices to have an $M$ and an $L$ so that
\begin{itemize}
\item $c_T = 0$ if $|T|>M$ and 
\item $c_T \leq L$
\end{itemize}
for each $T$.  We call this the \emph{independent} event construction, since an independent point process is assigned to each transition.  There may be other constructions, which we now describe.\\

Recall that a coupling is a probability space on which two or more processes are defined.  In the event construction, realizations of the process starting from any set of initial configurations are already coupled, since they are determined from the same spacetime event map.\\

For a multi-type particle system, the independent event construction leads to a certain coupling of realizations.  Using the same point processes to determine some transitions, we can obtain other couplings.\\

\begin{definition}
An \emph{event coupling} for a multi-type particle system is an event structure $\{e_i\}$, $\{r_i\}$, $e_i:F^{T_i}\rightarrow F^{T_i}$ such that
\begin{itemize}
\item corresponding to each transition $c_T(\phi,\psi)$ is a subcollection $I_c$ of indices $i$ such that
\begin{enumerate}
\item $T_i\supseteq T$,
\item $e_i(\phi')|_T = \psi$ and $e_i(\phi')|_{T_i - T} = \phi'|_{T_i - T}$ whenever $\phi'|_T = \phi$, and
\item $c_T(\phi,\psi) = \sum_{i \in I_c} r_i$
\end{enumerate}
\item aside from this, $e_i(\phi)=\phi$, and
\item if $e_i$ is assigned to the transitions $c_{T_j}(\phi_j,\psi_j)$, $j=1,2,...$ then
\begin{equation}\label{ecc}
\{\eta:\eta|_{T_j} = \phi_j\} \cap \{\eta:\eta|_{T_k} = \phi_k\} = \emptyset,\,j\neq k
\end{equation}
\end{itemize} 
\end{definition}
In the context of an event coupling the $e_i$ are called \emph{transition mappings}.  The first condition is so that the correct transitions occur, and at the correct rate.  The second condition is so that nothing else happens.  The third condition is to ensure that we respect the joint distribution of distinct transitions in the process.\\

It follows from the condition in \eqref{ecc} that coupled transitions must occur on regions $T_j$ that intersect in at least one site, since if $T_j \cap T_k = \emptyset$ there are configurations $\eta$ satisfying both $\eta|_{T_j} = \phi_j$ and $\eta|_{T_k} = \phi_k$.  Using this fact and a topological assumption on $S$ we can get a bit more control on event couplings.\\

\begin{proposition}
Suppose that $S$ is the set of vertices of an undirected graph with bounded degree and that transitions $c_T(\phi,\psi)$
\begin{itemize}
\item satisfy the conditions of Lemma \ref{bp} and
\item occur only on connected subsets $T$
\end{itemize}
Then in any event coupling, each mapping is associated to at most a uniformly bounded number of distinct transitions.
\end{proposition}
\begin{proof}
From the overlap condition described above, it follows from connectedness that each mapping can be assigned to at most finitely many transitions $c_{T_j}(\phi_j,\psi_j)$, in particular those involving sites within graph distance $M$ of $T_1$.  If $d+1$ is the degree of the graph (and $M\geq 2$, say), then there are at most $M(d^{M+1}-1)/(d-1)$ such sites (and thus a uniformly bounded number of distinct possible transitions involving these sites), since a rooted tree with branching number $d$ and depth $M$ has $1+d+...+d^M$ vertices.\\
\end{proof}
We note that from any valid event coupling we obtain a coupling of realizations of our process starting from any initial configuration.  In certain cases, some of these couplings may be more useful than others, for example if the resulting evolution of the system respects a certain natural partial order on the states, as we now describe.\\

\section{Attractiveness and additivity}\label{secattr}
Suppose the set of types $F$ is a partially ordered set, that is, we have a reflexive transitive relation $\leq$ defined for certain pairs $a,b \in F$.  From this we obtain a partial order on local (or global) configurations given by $\phi \leq \phi' \Leftrightarrow \phi(x)\leq\phi'(x)$ for each $x \in T$ ($\eta\leq\zeta \Leftrightarrow \eta(x)\leq\zeta(x)$ for each $x\in S$).

\begin{definition}
A mapping $e:F^T\rightarrow F^T$ is \emph{attractive} if $\phi\leq\phi' \Rightarrow e(\phi)\leq e(\phi')$.  An event coupling is attractive if each of its mappings is attractive.  A process is attractive if it has an attractive event coupling.
\end{definition}
Note that attractiveness of a process is equivalent to the existence of an event coupling so that $\eta_0\leq\eta_0' \Rightarrow \eta_t\leq\eta_t'$ for all $t>0$.\\

The historical reason for the word ``attractive'' is given in \cite{ips}, Chapter 2, Section 2; in that context, the ``spins'' at adjacent sites tend to align, which means that the spin state at one site tends to be attracted to the spin state at adjacent sites.\\

In some cases $F$ is equipped with a \emph{join} operation $\vee$, which is a symmetric binary operation defined for all pairs $a,b\in F$ that satisfies $a\vee b \geq a,b$ and $c\geq a,b \Rightarrow c\geq a\vee b$.\\

\begin{definition}\label{adddef}
A mapping $e:F^T\rightarrow F^T$ is \emph{additive} if $e(\phi\vee\phi') = e(\phi)\vee e(\phi')$.  An event coupling is additive if each of its mappings is additive, and a process is additive if it has an additive event coupling.
\end{definition}
Note that additivity of a process is equivalent to the existence of an event coupling so that $\eta_t = \eta_t'\vee\eta_t''$ whenever $\eta_0 = \eta_0'\vee\eta_0''$.\\

If $a$ and $b$ are \emph{comparable} i.e., $a\leq b$ or $b \leq a$, then $a\vee b$ is just $\max (a,b)$.  From this it follows that if $\eta_0\leq\eta_0' = \eta_0'$, then $\eta_0\vee\eta_0' = \eta_0'$ and so $\eta_t \vee\eta_t' = \eta_t'$ and in particular, $\eta_t \leq \eta_t'$.  Thus additivity implies attractiveness.\\

\section{Growth models and duality}\label{secgrowth}
In order to say more about additivity, we focus on the following class of systems, which are natural choices for either population growth or disease spread.  For what follows let $\underline{0}$ denote the configuration with all $0$'s.\\

\begin{definition}
A \emph{growth model} is a multi-type particle system with a distinguished \emph{passive type} $0$ and \emph{active types} $\{1,...,N\}$ satisfying
\begin{itemize}
\item $0 \leq a$ for each $a\in F$,
\item $\underline{0}$ is absorbing, and
\item $\underline{0}$ is reachable from any configuration with only finitely many active sites.
\end{itemize}
\end{definition}
If we think of the active types as representing organisms in various stages of development this means that vacancy is the ``lowest'' type, that no spontaneous birth occurs, and that a finite population has a chance of dying out.  For a population model, unoccupied territory is the passive type, and for a disease model in which each site represents a particular individual, the passive type might be a healthy individual.\\

For a growth model we are interested in questions of \emph{survival}, that is, whether the process started from a finite number of active sites has a chance of avoiding the $\underline{0}$ state for all time.\\

We will show that each additive growth model has an additive counterpart going backwards in time; the first step is to define its configurations.

\begin{definition}
For a growth model with types $F$ and join $\vee$, a set of active types $E\subseteq F$ is
\begin{itemize}
\item \emph{increasing} if $a \in E, a \leq b \Rightarrow b \in E$, and is
\item \emph{decomposable} if $a\vee b \in E$ implies $a \in E$ or $b \in E$.
\end{itemize}
The \emph{dual types} $\tilde{F}$ are the set of increasing and decomposable sets of active types $E$, together with a passive type $0$, with the partial order $0 \leq E$ for each $E \in \tilde{F}$ and $E \leq E'$ if $E \subseteq E'$.
\end{definition}

A partially ordered set of dual configurations $\tilde{F}^S$ is defined in the same way as for $F^S$.  From the following proposition we see that the dual types come equipped with a join.\\
\begin{proposition}
The set of dual types has a join $\vee$ given by $E\vee E' = E\cup E'$.
\end{proposition}

\begin{proof}
If $E$ and $E'$ are increasing and $a \in E\cup E'$ then $b\geq a$ implies $b \in E\cup E'$, so $E\cup E'$ is increasing, and that if $E$ and $E'$ are decomposable and $b \vee c \in E\cup E'$ then $b \in E\cup E'$ or $c \in E\cup E'$, so $E\cup E'$ is decomposable; therefore, $E\cup E' \in \tilde{F}$ whenever $E,E' \in \tilde{F}$.  If $E'' \geq E,E'$ then $E \supseteq E,E'$ so $E\supseteq E\cup E'$, moreover $E\cup E' \geq E,E'$.\\
\end{proof}

In order to relate configurations in $F^S$ with dual configurations in $\tilde{F}^S$ we introduce the following compatibility relation.

\begin{definition}
A configuration $\eta \in F^S$ is \emph{compatible} with a dual configuration $\zeta \in \tilde{F}^S$, denoted $\eta \sim \zeta$, if for some $x$, $\eta(x)$ is an active type and $\eta(x) \in \zeta(x)$.  For local configurations $\phi \in F^T, \theta \in \tilde{F}^T$ this reads
\begin{equation*}
\phi \sim \theta \Leftrightarrow \phi(x) \in \theta(x)\neq 0 \textrm{ for some }x \in T
\end{equation*}
\end{definition}

When the $\neq 0$ is understood it is usually omitted.  With the compatibility relation comes the natural identification $\zeta \leftrightarrow \{\eta:\eta \sim \zeta\}$ (or, for local configurations, $\theta \leftrightarrow \{\phi:\phi \sim \theta\}$) of dual configurations with the set of configurations with which they are compatible.\\

A growth model with configurations in $F^S$ has a \emph{dual} if there is an event coupling for the model and a multi-type particle system with configurations $\tilde{F}^S$ so that in the event construction for the process, the following \emph{duality relation} holds:
\begin{equation*}
\eta_t \sim \zeta_0 \Leftrightarrow \eta_0 \sim \zeta_t
\end{equation*}
where the process $\eta_s$, $0\leq s \leq t$ is run forward in time on the spacetime event map from $0$ to $t$, and the dual $\zeta_s$, $0\leq s \leq t$, is run backward in time on the spacetime event map from $t$ to $0$ (so that $\zeta_s$ is on the timeline $t-s$), using the dual mappings given by
\begin{equation*}
\tilde{e}(\theta) = \{\phi \in F^T:e(\phi) \sim \theta\}
\end{equation*}
Note the identification $\theta \leftrightarrow \{\phi:\phi\sim\theta\}$ is assumed on the right-hand side of the equation.\\

\begin{remark}
Duality for particle systems (the use of relations pairing particle systems with other systems going backwards in time) comes in more than one form; the form described here is known as coalescing duality in Chapter III of \cite{ips}.  An application of this form of duality can be found in \cite{growth}, where it is used to help prove complete convergence of the growth models described in that paper.  More general forms of duality exist; see for example the paper \cite{dualfam} in which a form of duality is used for nonattractive contact processes.
\end{remark}

Returning to the task at hand (and to the definition of duality just given), we now prove that every additive growth model has an additive dual.

\begin{proof}[Proof of Theorem \ref{thmdual}]
Given the additive coupling of $\eta_s$, $0\leq s \leq t$, for all possible $\eta_0$, and given the dual configuration $\zeta_0$, set $\Lambda_0 = \{\eta:\eta \sim \zeta_0\}$ and 
\begin{eqnarray*}
\Lambda_s &=& \{\eta_{t-s}:\eta_t \sim \zeta_0\}\\
\Lambda_{s^+} &=& \{\eta_{t-s^+}:\eta_t\sim\zeta_0\}
\end{eqnarray*}
where $\eta_{t-s^+}$ is the configuration just preceding an event, if an event occurs at time $s$.  Suppose for the moment that $|S|<\infty$ so that events happen one at a time.  In the case $|S|=\infty$, we suppose the hypothesis of Lemma \ref{bp} are satisfied.  Then, for any finite $T\subset S$ and $t>0$, run the process on an increasing sequence of sets $T_i$ with $\cup_i T_i=S$, that is, for each $i$, construct the process using only the events that occur on $T_i$.  Then from the conclusion of Lemma \ref{bp} there exists an almost surely finite $i_0$ such that if $i\geq i_0$, then $\eta_s|_T$, $0\leq s \leq t$, run on $T_i$, is equal to $\eta_s|_T$, $0\leq s \leq t$, for the process itself.  The duality relation then follows from the proof in the finite case.\\

Say that a collection of configurations $\Lambda \subset F^S$ is \emph{dualizable} if there is a dual configuration $\zeta$, called the minimal configuration, so that
\begin{equation*}
\Lambda = \{\eta:\eta\sim\zeta\}
\end{equation*}
If such a $\zeta$ exists it is unique.  It is immediate that $\Lambda_0$ is dualizable.  Suppose that $\Lambda_s$ is dualizable and that there is an event at time $s$.  We have $\Lambda_{s^+} = \{\eta_{t-s^+}:\eta_{t-s} \in \Lambda_s\}$, so we examine the mapping $e:F^T\rightarrow F^T$ that corresponds to the event at time $s$.  Let $\zeta$ be the minimal configuration for $\Lambda_s$ and let $\theta$ be its restriction to $T$.  Then $\eta \in \Lambda_{s^+}$ if and only if $e(\phi) \sim \theta$ where $\phi$ is the restriction of $\eta$ to $T$.  Equivalently,
\begin{equation*}
\eta \in \Lambda_{s^+} \Leftrightarrow e(\phi)(y) \in \theta(y) \textrm{ for some } y\in T
\end{equation*}
For each $y\in T$ we have by additivity that $e(\phi)(y) = \vee_x e(\phi_x)(y)$ where $\phi_x(x) = \phi(x)$ and $\phi_x(y)=0$ if $y\neq x$.  Since $\theta(x)$ is decomposable it then follows that
\begin{equation*}
e(\phi)(y) \in \theta(y) \Leftrightarrow e(\phi_x)(y) \in \theta(y) \textrm{ for some }x \in T
\end{equation*}
We then set $\delta_x(a)$ to be the local configuration $\phi$ with $\phi(x)=a$ and $\phi(y)=0$ for $y\neq x$ and since
\begin{equation*}
e(\delta_x(a\vee b)) = e(\delta_x(a)) \vee e(\delta_x(b))
\end{equation*}
it follows that if $e(\delta_x(a\vee b))(y) \in \theta(y)$, then either $e(\delta_x(a))(y) \in \theta(y)$ or $e(\delta_x(b))(y) \in \theta(y)$, since $\theta(x)$ is decomposable.  By attractiveness and since $\theta(x)$ is increasing, if $\phi\leq\phi'$ and $e(\phi)(y)\in\theta(y)$ then $e(\phi')(y) \in \theta(y)$.  It follows that for each $x,y \in T$ the set
\begin{equation*}
\theta_{+,y}(x) = \{a \in F:e(\delta_x(a))(y) \in \theta(y)\}
\end{equation*}
is increasing and decomposable, thus is a dual type.  Thus, setting
\begin{equation*}
\zeta_{s^+}(x) = \{a \in F:\delta_x(a) \in \theta(y) \textrm{ for some }y \in T\}
\end{equation*}
when $x \in T$ (and $\zeta_{s^+}(x) = \zeta_s(x)$ when $x \notin T$) we observe that $\zeta_{s^+}(x) = \cup_y \theta_{+,y}(x)$ is a dual type, since dual types are closed under unions, and  moreover that
\begin{equation*}
\eta \in \Lambda_{s^+} \Leftrightarrow \eta(x) \in \zeta_{s^+}(x) \textrm{ for some } x
\end{equation*}
or in other words, $\Lambda_{s^+}$ is dualizable, and $\zeta_{s^+}$ is its minimal configuration.  This completes the induction step and shows that an additive process has a dual which is defined in the manner just described.  Dual mappings $\tilde{e}:\tilde{F}^T\rightarrow\tilde{F}^T$ can be given as
\begin{equation*}
\tilde{e}(\theta)(x) = \{a \in F:e(\delta_x(a)) \sim \theta\}
\end{equation*}
with the property that $e(\phi)\sim\theta \Leftrightarrow \phi\sim\tilde{e}(\theta)$.  We observe that $\phi \sim \tilde{e}(\theta\vee\theta') \Leftrightarrow e(\phi)\sim\theta \vee \theta' \Leftrightarrow e(\phi)(x) \in \theta(x) \cup \theta(x) \textrm{ for some }x \Leftrightarrow e(\phi)\sim \theta \textrm{ or }e(\phi)\sim \theta' \Leftrightarrow \phi\sim\tilde{e}(\theta) \textrm{ or }\phi\sim\tilde{e}(\theta') \Leftrightarrow \phi\sim\tilde{e}(\theta)\vee\tilde{e}(\theta')$, which shows that the dual is additive.  By definition of the dual types and the duality relation, it follows that the dual is a growth model.\\
\end{proof}
In order to push the duality relation backwards through an event it was necessary to decouple the interactions in the forward process, first by sites, then by type.  It was in order to decouple by type that we required dual types be decomposable.  In general, considerable complexity can arise, however in the following simple case we can easily identify the dual.\\

\begin{example}[An N-stage contact process]\label{ex1}
Take $F = \{0,1,...,N\}$ with primitive types $1,...,N$ totally ordered in the usual ordering of integers.  Let $G$ be an undirected graph, then define a process on $G$ using the transitions
\begin{itemize}
\item $i \rightarrow i+1$ at rate $\gamma$, for $i=1,...,N-1$, 
\item $i \rightarrow 0$ at rate $1$, for $i\geq 1$, and
\item $0 \rightarrow 1$ at rate $\lambda$ times the number of neighbours in state $N$.
\end{itemize}
The set of dual types is $0$ and $E_j = \{i \in F:i \geq j\}$, for $j=1,...,N$, with $E_j \subset E_k$ for $j>k$.  It is not hard to check that the dual transitions are
\begin{itemize}
\item $E_{i+1} \rightarrow E_i$ at rate $\gamma$, for $i=1,...,N-1$, 
\item $E_i \rightarrow 0$ at rate $1$, for $i\geq 1$, and
\item $0 \rightarrow E_N$ at rate $\lambda$ times number of neighbours in state $E_1$.
\end{itemize}
Introducing the correspondence $E_j \leftrightarrow N+1-j$, the set of dual types becomes $\{0,1,...,N\}$ with the usual order, and the transitions are identical to those in the original process.  Thus we see that this process is \emph{self-dual}.
\end{example}

\section{Primitivity and colour decomposition}\label{secprim}
The goal of this section is to prove Theorem \ref{thmadd}.  We will make our way towards multi-colour systems in a few steps.  First we need a couple of definitions.\\

\begin{definition}
For a growth model with a join, an active type $a \in F$ is \emph{primitive} if $a = b \vee c \Rightarrow a=b \textrm{ or }a=c$, and is \emph{compound} otherwise.  The primitive types are denoted $F_p$.
\end{definition}
It is possible to have $a,b \in F_p$ with $a<b$; this is true for example when $F$ is totally ordered, since in that case all active types are primitive.  We say that $a,b \in F$ are \emph{incomparable} or $a<>b$ if $a \nleq b$ and $b \nleq a$, which is the opposite of being comparable.\\

\begin{definition}
A set $C \subset F_p$ satisfying $a=b$ or $a<>b$ for $a,b \in C$ is called a \emph{colour combination}.  We denote by $\mathcal{C}$ the set of colour combinations.
\end{definition}
Notice that $\mathcal{C}$ is partially ordered by $C \leq C' \Leftrightarrow$ for each $a \in C$, there exists $b \in C'$ such that $b \geq a$.  There is also a join defined by
\begin{equation*}
C\vee C' = \{a \in C\cup C':a \not < b\textrm{ for any }b \in C\cup C'\}
\end{equation*}
\begin{definition}
For the set of types $F$ of a growth model with a join, the \emph{multi-colour expansion} $F_*$ is the set $\mathcal{C}$ of colour combinations of $F$, together with a passive type $0$ satisfying $0 \leq a$ for each $a \in F_*$.
\end{definition}
From the definition of the join on $\mathcal{C}$ it follows that the sets $\{a\}$, for $a \in F_p$, are the primitive types in $F_*$.\\

There is a mapping $\pi:F_*\rightarrow F$ defined by $\pi(0)=0$ and
\begin{equation*}
\pi:C \mapsto \bigvee_{a \in C}a
\end{equation*}
for $C \in \mathcal{C}$.  Notice that for $b \in F$, $\pi^{-1}(b) = \{C \in \mathcal{C}:b = \bigvee_{a \in C}a\}$.  If $\pi(C)=a$ we say that $C$ is a \emph{decomposition} of $a$, and if $a$ has a unique decomposition we use $C(a)$ to denote it.  Notice that $C(a\vee b) = C(a)\vee C(b)$, if they exist.
\begin{proposition}\label{pm}
The mapping $\pi:F_*\rightarrow F$ has the following properties:
\begin{itemize}
\item $a\leq b \Rightarrow \pi(a)\leq \pi(b)$.
\item $\pi(a\vee b) = \pi(a)\vee \pi(b)$.
\item $\pi$ is surjective.
\item $\pi^{-1}(a) = \{a\}$ for each $a \in F_p$.
\end{itemize}
\end{proposition}
\begin{proof}
The first and second assertions are clear.  To see that $\pi$ is surjective, first note that $F$ can be partitioned into $\{0\}$, primitive types, and compound types.  By definition, $\pi(0)=0$.  If $a \in F_p$, the set $\{a\}$ belongs to $\mathcal{C}$ and $\pi(\{a\}) = a$.  Finally, if $b \in F$ is compound, then by repeatedly breaking up joins, we see that $b$ can be written as
\begin{equation*}
b = \bigvee_{a \in C}a
\end{equation*}
for some $C \in \mathcal{C}$, and $\pi(C)=b$.\\

For each $a \in F_p$, from primitivity it follows that $a$ has no decomposition containing more than element, moreover $\pi(\{c\})\neq a$ if $c \neq a$.  Since $\pi(\{a\}) = a$ it follows that $\pi^{-1}(a)=\{a\}$.
\end{proof}
For future use we note that $\pi$ can be defined on configurations by letting $\pi(\eta)(x) = \pi(\eta(x))$ for each $x$.\\

From Proposition \ref{pm} it follows that
\begin{itemize}
\item each primitive type has exactly one decomposition,
\item each compound type has at least one decomposition, and
\item primitive types in $F$ \emph{lift} to primitive types in $F_*$ in the sense that
\begin{enumerate}
\item the sets $\{a\}$ are the primitive types in $F_*$,
\item $\pi^{-1}(a)=a$ for each $a \in F_p$ and
\item the order is preserved under $\pi^{-1}$
\end{enumerate}
\end{itemize}
From the lifting property we see that the sets of primitive types $F_{*p}$ and $F_p$ are isomorphic, $\pi$ being a natural isomorphism between them.\\

\begin{remark}
A compound type in $F$ can have more than one decomposition and thus more than one preimage under $\pi$, as in the case $F = \{0,1,2,3,4\}$ with $0<1,2,3<4$, so that $F_p = \{1,2,3\}$ and $4 = 1\vee 2 = 1\vee 3 = 2\vee 3= 1\vee 2\vee 3$.  Here $F_*$ can be labelled $\{0,1,2,3,4,5,6,7\}$ with $\pi(a)=a$ and $\pi^{-1}(a)=a$ for $a=0,1,2,3$, $4=1 \vee 2$, $5=1 \vee 3$, $6=2\vee 3$ and $7 = 1 \vee 2 \vee 3$, and with $\pi^{-1}(4) = \{4,5,6,7\}$.  If we add $1<2$ to the order on $F$ then $F_p$ is still $\{1,2,3\}$ but $4 = 1\vee 3 = 2\vee 3$ only (in the sense of colour combinations).  In this case $F_*$ can be labelled $\{0,1,2,3,4,5\}$ with $\pi(a)=a$ and $\pi^{-1}(a)=a$ for $a = 0,1,2,3$, $4=1\vee 3$ and $5=2\vee 3$, and with $\pi^{-1}(4) = \{4,5\}$.
\end{remark}

\begin{definition}\label{defmc}
A set of types $F$ with a join is \emph{multi-colour} if every compound type has exactly one decomposition.  If $F$ is multi-colour and $b \in F$, $C(b)$ denotes its decomposition.
\end{definition}
For any set of types $F$ with a join, its multi-colour expansion $F_*$ is multi-colour.  From Proposition \ref{pm}, $F$ being multi-colour is equivalent to $\pi:F_*\rightarrow F$ being injective, which is equivalent to $F_*$ and $F$ being isomorphic, $\pi$ being a natural isomorphism between them.  Thus for a multi-colour system there is a $1:1$ correspondence between the set of active types and the set of colour combinations.\\

Before proving Theorem \ref{thmlift}, we consider the following simple example of a multi-type system and its lift.\\

\begin{example}[A 3-type system]
Take the model with $F = \{0,1,2,3,4\}$ with incomparable primitive types $F_p = \{1,2,3\}$ and $4 = 1\vee 2 = 1\vee 3=2 \vee 3 = 1\vee 2 \vee 3$.  Each non-zero type dies at rate $1$, and produces an individual of type $1,2$ or $3$ at rate $\lambda/3$.  Then the lift has types $F_* = \{0,1,2,3,1\vee 2, 1\vee 3,2\vee 3,1\vee 2 \vee 3\}$, and the transitions are the same:  each non-zero type dies at rate $1$, and produces an individual of type $1,2$ or $3$ at rate $\lambda/3$.  The only difference is that if, for example, a $2$ lands on a site occupied by a $3$, then in the lift the resulting type is $2\vee 3$ instead of just $4$; in the lift we distinguish the various joins of primitive types.
\end{example}

We now prove Theorem \ref{thmlift}.  For the proof, define the \emph{minimal} elements $\min(E)$ of a set $E$ to be the elements $a \in E$ such that $a \not >b$ for any $b \in E$.  Each set $E$ has a \emph{layer partition} $E_1,E_2,...,E_k$ given by $E_1 = \min(E)$ and $E_i = \min(E - E_{i-1})$ for $i>1$.\\

\begin{proof}[Proof of Theorem \ref{thmlift}]
We take transition mappings $e:F^T\rightarrow F^T$ to new mappings
\begin{equation*}
e_*:F_*^T\rightarrow F_*^T
\end{equation*}
Having done this, $\xi_t$ is constructed using the transitions $e_*$.  To ensure that $\pi(x_0)=\eta_0 \Rightarrow \pi(\xi_t)=\eta_t$ for $t>0$ it suffices to show that $\pi\circ e_* = e\circ \pi$.  Since $F_{*p}$ and $F_p$ are isomorphic we use $a$ to denote both $a \in F_{*p}$ and $\pi(a) \in F_p$.\\

Partition $F_p$ into layers $F_{p1},...,F_{pk}$.  For a set $B$ let $\vee B$ denote $\vee_{b \in B}b$.  For $a \in F_1$ and $x \in T$ let
\begin{equation*}
e_*(\delta_x(a)) = \vee\pi^{-1}(e(\delta_x(a)))
\end{equation*}
Since for $a,b \in F_*$, $\pi(a)=\pi(b) \Rightarrow \pi(a\vee b)=\pi(a)\vee \pi(b) = \pi(b)$, it follows that $\pi(e_*(\delta_x(a)) = e(\delta_x(a))$.  To extend $e_*$ to other layers in an additive way we first make the following observations.  If $a < a'$ then $a' = a \vee a'$ and
\begin{equation*}
e(\delta_x(a')) = e(\delta_x(a\vee a')) = e(\delta_x(a)\vee\delta_x(a')) = e(\delta_x(a))\vee e(\delta_x(a'))
\end{equation*}
by additivity of $e$.  If $\phi' = \phi\vee\phi'$ and $\psi \in \pi^{-1}(\phi)$, $\psi' \in \pi^{-1}\phi'$ then 
\begin{equation*}
\pi(\psi\vee\psi')=\pi(\psi)\vee\pi(\psi')=\phi\vee\phi' = \phi'
\end{equation*}
Thus, for $i>1$ suppose $e_*$ is defined on $\delta_x(a)$ for each $x \in T$, $a \in F_{pj}$ and $j<i$.  Then let
\begin{equation*}
e_*(\delta_x(a')) = \vee\pi^{-1}(e(\delta_x(a'))) \vee \bigvee\{e_*(\delta_x(a)):a<a',a \in F_{pj},j<i\}
\end{equation*}
It follows from the last observation that $\pi(e_*(\delta_x(a')))=e(\delta_x(a'))$, moreover $e_*(\delta_x(a)\vee\delta_x(a')) = e_*(\delta_x(a'))$ for any $a<a',a \in F_{pj}$ for some $j<i$.  Doing this for each $i$, we find that $e_*(\delta_x(a))$ is defined for each $x \in F^T$ and each $a \in F_p$, and is additive on that set.\\

For each compound type $b \in F_*$ and $x \in T$, let
\begin{equation*}
e_*(\delta_x(b))=\bigvee\{e_*(\delta_x(a)):a \in C(b)\}
\end{equation*}
Finally, for arbitrary $\phi \in F_*^T$ let $e_*(\phi) = \bigvee_x e_*(\phi_x)$ where $\phi_x(x) = \phi(x)$ and $\phi_x(y)=0$ for $y\neq x$.  Both extensions preserves additivity, and so $e_*$ is defined and additive on all of $F_*^T$.\\

To check that $\pi(e_*(\phi)) = e(\pi(\phi))$ observe that $\pi(e_*(\delta_x(a))) = e(\pi(\delta_x(a)))$ for $a \in F_p$ and so
\begin{eqnarray*}
\pi(e_*(\phi)) &=& \pi(\vee e_*(\delta_x(a)))\\
&=& \vee\pi(e_*(\delta_x(a)))\\
&=& \vee e(\pi(\delta_x(a)))\\
&=& e(\pi(\vee\delta_x(a)))\\
&=& e(\pi(\phi))
\end{eqnarray*}
where the $\vee$ is over $x\in T$ and $a \in C(\phi(x))$.
\end{proof}

\begin{remark}
If $\xi_t$ is the lift of $\eta_t$ and $\eta_t = \underline{0}$ then $\xi_t = \underline{0}$, since the configuration with all $0$'s has only itself as a preimage under $\pi$.  Therefore $\eta_t$ survives, i.e., $\eta_t \neq \underline{0}$ for all $t>0$, if and only if $\xi_t$ survives.  For this reason, when studying questions of survival (which are usually the interesting questions for growth models), it is sufficient to restrict our attention to multi-colour systems.
\end{remark}

Having lifted to a multi-colour system, we now consider duality in that context, for which the following fact is helpful.
\begin{lemma}\label{mcpd}
If $F$ is multi-colour and $a \in F$ is primitive then
\begin{equation*}
a \leq b\vee c \Rightarrow a \leq b \textrm{ or } a \leq c
\end{equation*}
\end{lemma}

\begin{proof}
It suffices to think in terms of colour combinations.  If $a \leq b \vee c$ then $a \leq a'$ for some $a' \in C(b\vee c)$.  Since $C(b\vee c) = C(b)\vee C(c) \subset C(b)\cup C(c)$, $a' \in C(b)$ or $a' \in C(c)$, so that $a\leq b$ or $a \leq c$.
\end{proof}

If $F$ is not multi-colour this may not hold, as in the case $F = \{0,1,2,3,4\}$ with $0<1,2,3<4$, in which case $3 \leq 4 = 1 \vee 2$ but $3 \nleq 1,2$.\\

From Lemma \ref{mcpd} we conclude that for a multi-colour system the set
\begin{equation*}
E_a = \{a \in F:b \geq a\}
\end{equation*}
is not only increasing, but also decomposable, when $a$ is primitive.  Thus for a multi-colour growth model and each primitive $a$, $E_a$ is a dual type.  We show next that these are exactly the primitive dual types.\\

We say that $C(c) \succ C(b)$ if $a' \geq a$ for some $a' \in C(c)$, $a \in C(b)$.  Note that $\succ$ is reflexive but not necessarily antisymmetric or transitive; for example, let $b_i = a_i\vee a_{i+1}$ for some set of incomparable $a_i$, then $C(b_i)\succ C(b_j) \Leftrightarrow |i-j|\leq 1$.\\

Define also $C \sqcup C' = \{a \in C \cup C':a \not > b \textrm{ for any }b \in C\cup C'\}$, or $=0$ (the passive type) if $C \cap C'=\emptyset$.\\

\begin{proposition}\label{pdt}
If $F$ is multi-colour then $\tilde{F}$ is multi-colour with primitive types
\begin{equation*}
\tilde{F}_p = \{E_a:a \in F_p\}
\end{equation*}
and $F$ identifies with $\tilde{F}$ on active types via
\begin{equation*}
b \leftrightarrow E_b = \{c \in F:C(c)\succ C(b)\}
\end{equation*}
Moreover,
\begin{equation*}
E_a \vee E_b = \{c \in F:C(c) \succ C(a)\sqcup C(b)\}
\end{equation*}
In particular, the identification $a \leftrightarrow E_a$ of $F_p$ with $
\tilde{F}_p$ is order-reversing.
\end{proposition}

\begin{proof}
We first exhibit the identification of $F$ with $\tilde{F}$, then deduce the rest.  For $E\in\tilde{F}$ define the \emph{minimal types}
\begin{equation*}
E_m = \min(E) = \{a \in E:a \not > b\,\, \forall b \in E\}
\end{equation*}
Each minimal type $a \in E_m$ is primitive, since if $a=b\vee c$ then since $E$ is decomposable, $b \in E$ or $c \in E$; if $b \in E$ then $a \not > b$, but since $a \geq b$ it follows that $a=b$.  Moreover, $a<>b$ for $a,b \in E_m$ thus $E_m$ is a colour combination.\\

Since $E$ is increasing, $E \supseteq \{b:b \geq a \textrm{ for some } a \in E_m\} = \bigcup_{a \in E_m}E_a$, moreover $b \in E \Rightarrow b \in E_m \textrm{ or } b>a \textrm{ for some } a \in E_m$ which means $E \subseteq \bigcup_{a \in E_m}E_a$, thus $E = \bigcup_{a \in E_m}E_a$.  Therefore $E = \{c \in F:C(c) \succ E_m\}$, or $E = E_b = \{c \in F:C(c) \succ C(b)\}$ where $b \in F$ is the unique active type that satisfies $C(b)=E_m$.  To each active type $b$ corresponds a different set $E_b$, which yields the identification of $F$ with $\tilde{F}$.  The expression for $E_a \vee E_b$ follows easily.\\

Observe that if $a \in F_p$ and $a \in C(b)\sqcup C(c)$ then either $a \in C(b)$ or $a \in C(c)$, so that if $E_a = E_b \vee E_c$ then $a \in \min E_b$ or $a \in \min E_c$.  Supposing the former, $E_a \subset E_b$ which means $E_a =E_b$, i.e., $E_a$ is primitive.  On the other hand, if $b$ is compound then $C(b) = \sqcup_{a \in C(b)} C(a)$, so that $E_b = \vee_{a \in C(b)}E_a$ and $E_b$ is not primitive.  Thus $\tilde{F}_p = \{E_a:a \in F_p\}$.\\

Since $E_b \neq E_c$ when $C(b)\neq C(c)$, the dual is multi-colour.  If $a,b \in F_p$ are primitive and $a < b$ then $C(a) \sqcup C(b) = C(a)$, so $E_a > E_b$.\\
\end{proof}

We now approach the \emph{double dual}.  If a growth model has a dual, then from the proof of Theorem \ref{thmdual} it follows that its dual is additive, and so the dual has a dual, which is the double dual, whose set of types we denote by $\overline{F}$, using $\lambda$ to denote an element of $\overline{F}$.  For a multi-colour system, it follows from an application of Proposition \ref{pdt} to $\tilde{F}$ that the double dual is multi-colour and has primitive types
\begin{equation*}
\lambda_a = \{E \in \tilde{F}:E\supseteq E_a\}
\end{equation*}
where $a \in F_p$, which are in $1:1$ correspondence with $F_p$.  Since duality is order-reversing on primitive types, double duality is order-preserving on primitive types.  By identifying active types with colour combinations, we obtain an order-preserving identification of $F$ with $\overline{F}$ given by
\begin{equation*}
b \leftrightarrow \lambda_b = \{E \in \tilde{F}:E \supseteq E_a \textrm{ for some } a \in C(b)\}
\end{equation*}
For a local configuration $\phi \in F^T$, let $\Xi_{\phi} \in \overline{F}^T$ denote the corresponding local double dual configuration, under the above identification.  The following relationship holds for compatibility.
\begin{lemma}\label{ddc}
Let $\phi \in F^T$ and $\theta \in \tilde{F}^T$, then
\begin{equation*}
\phi \sim \theta \Leftrightarrow \theta \sim \Xi_{\phi}
\end{equation*}
\end{lemma}
\begin{proof}
It is sufficient to have $b \in E \Leftrightarrow E \in \lambda_b$, for each $b \in F$ and $E \in \tilde{F}$, and it is not hard to check that $b\in E_c
\Leftrightarrow C(b)\succ C(c)
\Leftrightarrow \exists a \in C(b): C(a)\succ C(c)
\Leftrightarrow C(a)\sqcup C(c) = C(c)
\Leftrightarrow E_a \vee E_c = E_c
\Leftrightarrow E_c \supseteq E_a
\Leftrightarrow E_c \in \lambda_b$.
\end{proof}
We now prove Theorem \ref{thmadd} by showing that if an additive multi-colour growth model has a dual, then it identifies with its double dual, which since the double dual is additive, implies that the process itself is additive.
\begin{proof}[Proof of Theorem \ref{thmadd}]
It suffices to show that transitions in the double dual commute with the identification of $F^S$ with $\overline{F}^S$, i.e., to show that for double dual transition mappings $\overline{e}$,
\begin{equation*}
\overline{e}(\Xi_{\phi}) = \Xi_{e(\phi)}
\end{equation*}
Transition mappings $\overline{e}$ of the double dual are given by
\begin{equation*}
\overline{e}(\Xi) = \{\theta : \tilde{e}(\theta) \sim \Xi\}
\end{equation*}
where the identification $\Xi \leftrightarrow \{\theta: \theta \sim \Xi\}$ is implicit.  However, $\tilde{e}(\theta)\sim \Xi_{\phi} \Leftrightarrow \phi \sim \tilde{e}(\theta) \Leftrightarrow e(\phi) \sim \theta \Leftrightarrow \theta \sim \Xi_{e(\phi)}$, using Lemma \ref{ddc} and the duality relation for the process on $F^S$.  Therefore,
\begin{equation*}
\overline{e}(\Xi_{\phi}) = \{\theta: \theta \sim\Xi_{e(\phi)}\} = \Xi_{e(\phi)}
\end{equation*}
as desired.
\end{proof}

\section{Percolation viewpoint}\label{secperc}
A percolation model (percolation process) is any stochastic model with some sort of spatial structure and a notion of \emph{path} from one point to another.  We distinguish three models in increasing order of resemblance to a particle system.
\begin{enumerate}
\item
The classical percolation model \cite{perc}: take the random subgraph of the graph $(V,E)$ with $V = \mathbb{Z}^2$ and $E = \{xy:\|x-y\|_{\infty}=1\}$ in which each edge is included independently with probability $p \in (0,1)$.  Sites $x,y \in V$ are \emph{connected} if they are linked by a path of edges.  Percolation occurs if $(0,0)$ belongs to an infinite cluster of connected sites.
\item
Oriented percolation in two dimensions \cite{oriperc}:  take the random subgraph of the directed graph $(V,E)$ with $V = \{(x,y) \in \mathbb{Z}^2:y \geq 0, x+y \textrm{ is even}\}$ and edges from $(x,y)\rightarrow (x-1,y+1)$ and from $(x,y)\rightarrow (x+1,y+1)$ for each $(x,y) \in V$, again with each edge included independently with probability $p$.  Percolation occurs if there is an infinite path of directed edges starting at $(0,0)$.
\item
The contact process (this one \emph{is} a particle system):  take the particle system with types $\{0,1\}$ and transitions
\begin{itemize}
\item $0 \rightarrow 1$ at site $x$ at rate $\lambda n_1(x)$ and
\item $1\rightarrow 0$ at each site at rate 1
\end{itemize}
where $n_1(x)$ is the cardinality of the set $\{y:xy\in E, y \textrm{ is in state 1 }\}$.  The spacetime event map for the process leads to a percolation model on the spacetime set $V \times \mathbb{R}^+$ in which the interval $\{x\}\times [s,t]$ is an upward directed edge if there are no deaths on the interval, and $(x,t) \rightarrow (y,t)$ is a directed edge whenever there is a transmission event from $x$ to $y$ at time $t$.  Percolation occurs from $(x,0)$ if there is an infinite path (i.e., a path of infinite length) of directed edges starting at $(x,0)$.
\end{enumerate}
For the contact process it is not hard to confirm that percolation from $(x,0)$ implies survival starting from the single active site $x$, and vice versa.  For a multi-particle system it may seem less clear how to prescribe a percolation process.  However, for an additive multi-colour system it can be easily achieved, provided that we colour code the edges, as we show now.\\

On the spacetime set $V \times \mathbb{R}^+$, the interval $\{x\}\times [s,t)$ is an edge of every colour if there are no events on that interval, since any colour can propagate upwards along it.  If an event occurs at time $t$ in the transition mapping $e:F^T\rightarrow F^T$, there is an edge from $(x,t^-) \rightarrow (y,t)$ (where $x=y$ is allowed) with initial colour $a$ and terminal colour $b$ if $b \in C(e(\delta_x(a))(y))$, that is, if type $a$ at $x$ \emph{produces} type $b$ at $y$ through $e$.  Edges $\{x\}\times [s,t)$ with colour $a$ and $\{y\}\times [t,u)\}$ with colour $b$ can be linked by an edge from $(x,t^-) \rightarrow (y,t)$ if it has initial colour $a$ and terminal colour $b$.  Percolation occurs from type $a$ at $(x,0)$ if there is an infinite path of directed coloured edges starting at $(x,0)$ with colour $a$, and it characterizes survival:
\begin{theorem}
For an additive multi-colour particle system, survival occurs starting from the single type $a$ active site $x$ if and only if percolation occurs from type $a$ at $(x,0)$.
\end{theorem}
\begin{proof}
If there is a $t>0$ so that all coloured paths from $(x,0)$ are contained in $V \times [0,t)$ then it follows by additivity that the process has died out by time $t$.  If there is an infinite path, then it follows again by additivity that the process survives.
\end{proof}

\section{Population viewpoint}\label{secpop}
Multi-colour systems can be understood in a natural way as models of interacting populations.  Let each colour correspond to a type of organism (plant, animal, fungus etc.).  Recalling what it means to be a colour combination, at each site we can have up to one each of any collection of organisms of different types, provided the types are incomparable.  Thus the restrictions on the population are that
\begin{itemize}
\item there can be at most one organism of a given type, at a given site, and
\item organisms of comparable type cannot exist together at the same site.
\end{itemize}
These are realistic assumptions, if we think of populations having a carrying capacity, and of stronger types excluding weaker types.  If we wanted to model a carrying capacity $K>1$ we could just assign a primitive type to each number $1,2,3,...,K$ of organisms of a certain type and then order them $1<2<...<K$.\\

At an event with transition mapping $e:F^T\rightarrow F^T$, for $a,b\in F_p$ and $x,y \in T$ we say that $a$ at $x$ \emph{produces} $b$ at $y$, denoted $e:(a,x)\rightarrow (b,y)$, if
\begin{equation*}
b \in C(e(\delta_x(a))(y))
\end{equation*}
that is, if $b$ is one of the organisms in $e(\phi)(y)$, when $\phi$ is the local configuration having only the organism $a$ at site $x$ and no other organisms at sites $y \in T, y \neq x$.\\

Using this notion of production, we can define a multi-type branching process in which each individual gives birth to individuals of the appropriate types at the appropriate rates.  We can then ask how the spatial process differs from the branching process, which is answered by the following result.

\begin{theorem}
For an additive multi-colour system, the only interactions between organisms are due to the effects of crowding and exclusion.
\end{theorem}
\begin{proof}
For an additive system, to compute the effect of a transition we first focus on each organism $(x,a)$ present before the event and include everything that it produces.  Then, if for example $(x,a)$ produces $(z,c)$ but $(y,b)$ produces $(z,d)$ with $d>b$, organism $d$ \emph{excludes} organism $b$; if $(y,b)$ also produces $(z,b)$ then only one copy of organism $b$ is retained at $z$.  Thus, accounting for each organism's production and then clearing any duplicate or excluded organisms gives the result of a transition.
\end{proof}

With this characterization in hand we now describe which types of interaction fail to be additive:
\begin{enumerate}
\item
If the system exhibits \emph{inhibition}, that is, the presence of one type of organism decreases the production of another organism, then it cannot be additive, or even attractive; here's why.  The introduction of an organism increases the configuration.  However, for some organism its introduction decreases the production of another organism, so we have a transition $e$ with $\phi > \phi'$ but $e(\phi)<e(\phi')$.\\

\item
If the system exhibits \emph{cooperation}, that is, multiple organisms can produce together what they could not produce alone, then $e(\vee_{x,a}\delta_x(a)) > \vee_{x,a}e(\delta_x(a))$ for some collection of organisms $(x,a)$.  Thus the system may be attractive, but cannot be additive.\\
\end{enumerate}
In an additive system, organisms neither inhibit nor enhance each other's production.  In this sense, additive systems are the least interactive of all interacting particle systems.  However, mathematically, they are relatively easy to analyze and have fairly strong properties, and biologically, they lie on the boundary between inhibition and cooperation, and make good models of situations where, to first order of approximation, the only interaction between organisms arises from the fact that they take up space.

\section{Positive Correlations}\label{secPC}
On a partially ordered state space, we say a function $f$ is \emph{increasing} if $\eta \leq \eta' \Rightarrow f(\eta)\leq f(\eta')$, and is  \emph{decreasing} if $-f$ is increasing.  For a particle system, increasing functions include indicators of the events $\{\eta(x)\geq a\}$ for some $a,x$.\\

There is a useful property called positive correlations or PC which is said to hold for a measure $\mu$ on a partially ordered state space when \begin{equation*}
\mathbb{E}_{\mu}fg \geq \mathbb{E}_{\mu}f \mathbb{E}_{\mu}g
\end{equation*}
for all increasing functions $f,g$, provided the expectations are defined.  It can be checked, for example, that deterministic configurations, that is, measures concentrated on a single configuration, have positive correlations.  We say a particle system has positive correlations or PC if
\begin{equation*}
\textrm{PC at time }0 \Rightarrow \textrm{ PC at later times}
\end{equation*}
It is shown in \cite{ips}, Chapter II, that a particle system has PC if and only if all transitions are between comparable states, that is, for each $e,\phi$, either $e(\phi)\geq \phi$ or $e(\phi)\leq \phi$.\\

For an additive system, we can use this result to characterize PC in terms of production.  To warm up, we first discern a few types of production:
\begin{enumerate}
\item persistence:  $(x,a)$ produces $(x,a)$ and nothing else
\item movement:  $(x,a)$ produces $(y,a)$, for some $y\neq x$, and nothing else
\item birth:  $(x,a)$ produces $(x,a)$ and also something else
\item death:  $(x,a)$ produces nothing
\item promotion/demotion: $(x,a)$ produces $(x,b)$ with $b>a$ (promotion) or $b<a$ (demotion), and nothing else
\item death with dispersal:  $(x,a)$ produces nothing at $x$, but produces something elsewhere
\item neighbour-assisted survival:  $(x,a)$ produces nothing, but for some $y\neq x$, $(y,b)$ produces $(x,a)$
\item transmutation:  $(x,a)$ produces $(x,b)$ with $b<>a$
\end{enumerate}
Movement, for example, does not preserve PC, nor does death with dispersal or transmutation.  Persistence, birth, death, promotion/demotion and neighbour-assisted survival all preserve PC.\\

Say that an organism $(x,a)$ \emph{waxes} if $(x,a)$ produces at least $(x,a)$ (birth or persistence) and \emph{wanes} if it produces at most $(x,b)$ for some $b \leq a$ and nothing else (persistence, demotion or death).  A transition mapping $e$ is said to be \emph{productive} if each $(x,a)$ waxes, and \emph{destructive} if each $(x,a)$ wanes.  An organism $(x,a)$ \emph{compensates for the loss of} another organism $(y,b)$ in a transition if $(y,b)$ is demoted or dies, but $(x,a)$ produces $(y,b)$.
\begin{theorem}\label{thmPC}
For an additive multi-colour system, if each transition is either productive or destructive, then the system has PC.  A necessary and sufficient condition is that for each transition mapping, 
\begin{itemize}
\item each organism either waxes or wanes, and
\item if $(x,a)$ waxes then it compensates for the loss of every organism that wanes.
\end{itemize}
\end{theorem}
\begin{proof}
The sufficient condition follow easily from additivity, using the fact that $\phi\leq\psi,\phi'\leq\psi'\Rightarrow\phi\vee\phi' \leq \psi\vee \psi'$.  If a transition mapping has both a waxing organism $(x,a)$ that is not waning ($(x,a)$ produces both $(x,a)$ and something else, possibly $(x,b)$ for $b>a$ or $b<>a$, or possibly another organism elsewhere) and a waning organism $(y,b)$, and $(x,a)$ does not compensate for $(y,b)$'s loss, then letting $\phi = \delta_x(a) \vee\delta_y(b)$, $e(\phi)(y) < b = \phi(y)$ but $e(\phi)(z) > \phi(z)$ for some $z$, so $e(\phi) <> \phi$.  If the condition given does hold, then for any collection of organisms $(x,a)$ joined into a local configuration $\phi$, if at least one of them is waxing then $e(\phi) \geq \phi$ since all losses are compensated for by the waxing organism, and if they are all waning then $e(\phi)\leq \phi$.
\end{proof}
Note that PC is not always preserved by the dual, as the following example shows.\\

\begin{example}
Define the \emph{dandelion process} with types $\{0,1\}$ as follows.  For each $x \in V$, fix a dispersal distribution $p(x,\cdot)$ which is an atomic measure on subsets of $V$.  At each site $x$, at rate $1$, if $x$ is occupied then it dies and disperses, occupying the unoccupied sites in the set $A$ with probability $p(x,A)$.  Equivalently, for each site $x \in V$ and each subset $A \subset V$, at rate $p(x,A)$ (possibly $=0$), if $x$ is occupied then it dies and disperses to the sites in $A$.  This process is additive and multi-colour.  The dual to the dandelion process is the \emph{helper process} in which at rate $p(x,A)$, $x$ becomes (or remains) unoccupied unless some site in $A$ is occupied, in which case $x$ remains (or becomes)  occupied.  Since the dandelion process has death/dispersal transitions, it does not have PC.  In the helper process, however, all waning organisms have their losses compensated for, so it follows that the helper process has PC.
\end{example}

\section{Complete Convergence}\label{seccomp}

We focus now on additive multi-colour growth models having only productive and destructive transitions, which by Theorem \ref{thmPC} have positive correlations; for convenience we call these models \emph{simple}.  We first note the following fact.

\begin{lemma}
The dual of a simple model is simple (and thus has PC).
\end{lemma}
\begin{proof}
Reversing a productive transition (i.e., looking at the dual transition) gives a productive transition, and the same is true of a destructive transition.
\end{proof}

As in the Introduction, we say the model has single-site survival or more succinctly, \emph{survives} if, started from a single active site, with positive probability there are active sites for all time.  To achieve complete convergence it is sufficient that when the model survives, for any $\epsilon>0$ and appropriately rescaled, both the model and its dual dominate an oriented percolation process with parameter $>1-\epsilon$, and we can then use the strategy described in\cite{growth}.  To obtain the comparison to oriented percolation we show that the construction of \cite{crit} can be applied, with slight modifications, to any simple model that survives.  We then need to show that if a simple model survives, then so does it dual.  However, applying the construction in \cite{crit} to the model, if it survives then it has a non-trivial upper invariant measure, which using the duality relation shows that its dual survives.  Since the dual of the dual is the model itself and the dual is a simple model, the same argument goes the other way, so the model survives if and only if its dual survives, which implies (in fact is directly equivalent to the fact) that single-site survival and $\nu\neq\delta_0$ are equivalent, which is part of Theorem \ref{thmcc}; we now prove the rest of Theorem \ref{thmcc}.\\

\begin{proof}[Proof of complete convergence in Theorem \ref{thmcc}]
We outline the modifications to the construction in \cite{crit} but leave out the details, which would be cumbersome and somewhat redundant to write down explicitly.  First, we need the model to be attractive, translation invariant, symmetric, and have positive correlations (check).  We also want the model to be irreducible, for two reasons (although the first reason can be circumvented): i) so that from a single site we can produce a large finite disc of active sites with some probability and ii) so that when arguing for complete convergence we can ensure the model going forward in time can intersect the dual going backward in time (for this see Section 5 of \cite{growth}).  In \cite{crit} the construction is done for a nearest-neighbour model (the basic contact process) but this can be generalized to a finite-range model by widening the sides of the box to be at least as large as the range of the model, and correspondingly increasing the range of production of a large finite disc of active sites.  For a growth model the zero state is reachable from any finite configuration, so for any $h>0$, with positive probability an active site dies within time $h$ without spreading activity to other sites (by encountering the right sequence of destructive events - if it could not then it would not be a growth model), which allows us to reproduce the argument in \cite{crit} that shows that when the process survives for a long time, it has many active sites on the top and sides of a well-chosen rectangular box.  Beyond this, the arguments in \cite{crit} are geometrical in nature and no additional properties specific to the process are required.  To get complete convergence, if the model does not survive there is nothing to do.  If the model survives, then its dual also survives, and we apply the construction of \cite{crit} to both the model and its dual and using the resulting oriented percolation models, follow the approach of \cite{growth} to show complete convergence.
\end{proof}

\section{Examples}
We conclude with some examples of additive multi-type growth models.
\subsection{Two-Stage Contact Process}
This is a natural generalization of the contact process introduced in \cite{krone} and further studied in \cite{fox} in which (viewing the contact process as a model of population growth) there is an intermediate juvenile type that must mature before it can produce offspring.  There are three types $0,1,2$ and the transitions at a site $x \in V$ are
\begin{itemize}
\item $2\rightarrow 0$ at rate $1$
\item $1\rightarrow 0$ at rate $1+\delta$
\item $1\rightarrow 2$ at rate $\gamma$
\item $0\rightarrow 1$ at rate $\lambda n_2(x)$
\end{itemize}
where $n_2(x)$ is the cardinality of the set $\{xy \in E:y \textrm{ is in state }2\}$.  In order to obtain an additive process we take the event structure with transition mappings
\begin{itemize}
\item recovery of type $1$ and $2$ at each site at rate $1$
\item recovery of type $1$ at rate $\delta$
\item onset i.e., $1\rightarrow 2$ at rate $\gamma$
\item transmission along each edge at rate $\lambda$
\end{itemize}
In other words, it is enough to make sure that whenever a $2$ recovers, a $1$ also recovers.  It is then easy to check that the resulting process has the correct transition rates and is additive with respect to the join operation on types given by $a\vee b = \max a,b$, so the partial order is $0<1<2$.  Since each transition is either productive or destructive, the model preserves positive correlations.  Taking the model on $\mathbb{Z}^d$ with nearest-neighbour interactions, the conditions of Theorem \ref{thmcc} are satisfied, so complete convergence holds.  This was shown in \cite{fox}, with the proof given along the same lines but specifically for the two-stage contact process.

\subsection{Bipartite Infection Model}
Consider the following model with primitive types $\{0,m,f,\}$ where $m$ stands for male and $f$ stands for female.  We can construct a two-sex model of infection spread by defining the following transitions on primitive types:
\begin{itemize}
\item $m\rightarrow 0$ at rate $1$
\item $f\rightarrow 0$ at rate $1$
\item $m\rightarrow m\vee f$ at rate $\lambda$
\item $f\rightarrow m\vee f$ at rate $\lambda$
\item $0\rightarrow f$ at rate $\lambda n_m(x)$
\item $0 \rightarrow m$ at rate $\lambda n_f(x)$
\end{itemize}
where $n_f(x)$ is the cardinality of the set $\{xy:y \textrm{ has an }m\textrm{ type particle}\}$ and $n_m(x)$ is the cardinality of the set $\{xy:y \textrm{ has an }f\textrm{ type particle}\}$, and then extending transitions to the compound type $m\vee f$ in the obvious way (for example $m\vee f\rightarrow f$ and $m\vee f\rightarrow m$ each at rate $1$ from either $m$ or $f$ recovering).  Thus in this model we can imagine that at each site there is one male and one female each of which is either healthy or infectious, and males can only transmit to females, and females to males, either at the same site, or at neighbouring sites.  Taking the event construction in which each of the above transitions, at each site/edge, is assigned its own transition mapping, the model is additive with respect to the join with $0\vee m=m$, $0\vee f=f$ and $m\vee f$ assigned its own (compound) type, moreover each transition is either productive or destructive.  Taking the model on $\mathbb{Z}^d$ with nearest-neighbour interactions, the conditions of Theorem \ref{thmcc} are satisfied, so complete convergence again holds.

\subsection{Household Model}
Consider the following model introduced in \cite{konnohousehold}, which is somewhat similar to the $N$-stage contact process considered in Example \ref{ex1}.  The set of types is $\{0,1,...,N\}$, with parameters $\lambda,\gamma$ and the transitions are
\begin{itemize}
\item $i\rightarrow 0$ at rate $1$
\item $i\rightarrow i+1$ at rate $i\gamma$
\item $0\rightarrow 1$ at rate $\lambda n_N(x)$
\end{itemize}
where $n_N(x)$ is the cardinality of the set $\{xy:y \textrm{ is in state }N\}$.  They also examine the variant where $n_N(x)$ is replaced by $\lambda \sum_{xy \in E}\eta(y)$, the sum of the values at neighbouring sites.  In the first case, take the event construction in which at each site there is a transition mapping for recovery $i\rightarrow 0$ and one mapping for each of the $i\rightarrow i+1$ transitions, $i=1,...,N-1$, and along each edge, there is a mapping for transmission at rate $\lambda$.  In the second case take the same event construction except with $N$ transition mappings for transmission $e_1,...,e_N$ each at rate $\lambda$, with exactly the mappings $e_1,...,e_i$ being used when $\eta(y)=i$.  In both cases the resulting process is additive with respect to the join given by $a\vee b=\max a,b$, and each transition is either productive or destructive.  Taking the model on $\mathbb{Z}^d$ with nearest-neighbour interactions, the conditions of Theorem \ref{thmcc} are satisfied, so complete convergence again holds.

\section*{Acknowledgements}
The author's research is supported by an NSERC PGSD2 scholarship.

\bibliography{twostage}

\begin{thebibliography}{10}

\bibitem{crit}
C.~Bezuidenhout and G.~Grimmett.
\newblock The critical contact process dies out.
\newblock {\em Annals of Probability}, 18(4):1462--1482, 1990.

\bibitem{oriperc}
R.~Durrett.
\newblock Oriented percolation in two dimensions.
\newblock {\em Annals of Probability}, 12(4):999--1040, 1984.

\bibitem{growth}
R.~Durrett and R.H. Schonmann.
\newblock Stochastic growth models.
\newblock {\em Percolation Theory and Ergodic Theory of Infinite Particle
  Systems}, 8:85--119, 1987.

\bibitem{fox}
E.~Foxall.
\newblock New results for the two-stage contact process.
\newblock {\em To appear in J. Appl. Prob, preprint on the ArXiv}.

\bibitem{basic}
D.~Griffeath.
\newblock The basic contact processes.
\newblock {\em Stochastic Processes and their Applications}, 11(2):151--185,
  1981.

\bibitem{perc}
G.~Grimmett.
\newblock {\em Percolation}.
\newblock Springer, second edition, 1999.

\bibitem{gc}
T.E. Harris.
\newblock Additive set valued markov processes and graphical methods.
\newblock {\em Annals of Probability}, 6(3):355--378, 1978.

\bibitem{krone}
S.~Krone.
\newblock The two-stage contact process.
\newblock {\em Annals of Applied Probability}, 9(2):331--351, 1999.

\bibitem{ips}
T.M. Liggett.
\newblock {\em Interacting Particle Systems}.
\newblock Springer, 1985.

\bibitem{sis}
T.M. Liggett.
\newblock {\em Stochastic Interacting Systems: Contact, Voter and Exclusion
  Processes}.
\newblock Springer, 1999.

\bibitem{bookmtbp}
Charles~J Mode.
\newblock {\em Multitype branching processes: Theory and applications},
  volume~34.
\newblock American Elsevier Pub. Co., 1971.

\bibitem{mcp}
Claudia Neuhauser.
\newblock Ergodic theorems for the multitype contact process.
\newblock {\em Probability Theory and Related Fields}, 91(3-4):467--506, 1992.

\bibitem{dualfam}
A.~Sudbury.
\newblock Dual families of interacting particle systems on graphs.
\newblock {\em Journal of Theoretical Probability}, 13(3):695--716, 2000.

\bibitem{konnohousehold}
Nobuaki Sugimine, Naoki Masuda, Norio Konno, and Kazuyuki Aihara.
\newblock On global and local critical points of extended contact process on
  homogeneous trees.
\newblock {\em Mathematical biosciences}, 213(1):13--17, 2008.

\end{thebibliography}
\bibliographystyle{plain}
\end{document}